\documentclass{amsart}
\usepackage{amsmath,amsthm}
\usepackage{amsfonts,amssymb}

\usepackage{enumerate}

\newtheorem{thm}{Theorem}[section]
\newtheorem{cor}[thm]{Corollary}
\newtheorem{lem}[thm]{Lemma}
\newtheorem{prop}[thm]{Proposition}

\newtheorem{defn}[thm]{Definition}

\theoremstyle{remark}

\def\sph{\mathbb{S}^{d-1}}
\def\f{\frac}

 \def\a{{\alpha}}
 \def\b{{\beta}}
 \def\g{{\gamma}}
 \def\k{{\kappa}}
 \def\t{{\theta}}
 \def\l{{\lambda}}
 \def\d{{\delta}}
 \def\o{{\omega}}
 \def\s{{\sigma}}
 \def\la{{\langle}}
 \def\ra{{\rangle}}

 \def\tb{{\mathbf t}}

 \def\CF{{\mathcal F}}
 \def\CH{{\mathcal H}}

 \def\CP{{\mathcal P}}
 
 \def\CS{{\mathcal S}}
 \def\CT{{\mathcal T}}

 \def\CC{{\mathbb C}}
 
 \def\NN{{\mathbb N}}

 \def\RR{{\mathbb R}}

 \def\ZZ{{\mathbb Z}}
        
        \def\proj{\operatorname{proj}}

 \def\dd{\mathrm{d}}
 \def\ii{\mathrm{i}}
 \def\ee{\mathrm{e}}

\def\f{\frac}

\begin{document}

\title[Intertwining operator associated to symmetric groups]
{Intertwining operator associated to symmetric groups and summability on the unit sphere}
 
\author{Yuan Xu}
\address{Department of Mathematics\\ University of Oregon\\
    Eugene, Oregon 97403-1222.}\email{yuan@uoregon.edu}

\date{\today}

\begin{abstract} 
An integral representation of the intertwining operator for the Dunkl operators associated with symmetric 
groups is derived for the class of functions of a single component. The expression provides a closed
form formula for the reproducing kernels of $h$-harmonics associated with symmetric groups when one 
of the components is a coordinate vector. The latter allows us to establish a sharp result for the Ces\`aro 
summability of $h$-harmonic series on the unit sphere. 
\end{abstract} 

\maketitle

\section{Introduction}
\setcounter{equation}{0}

Associate to a reflection group $G$, the Dunkl operators are a family of commuting first order differential-difference 
operators that act on smooth functions on $\RR^d$ \cite{D89}. In the case that $G = S_d$, the symmetric group
of $d$ elements, these operators are defined by 
\begin{equation}\label{eq:dunkl}
 D_i f(x)  =\frac{\partial}{\partial x_i} f(x) + \k  \sum_{ \substack{ j=1\\j \ne i }}^d \frac{f(x) - f(x(i,j))}{x_i-x_j},
  \qquad 1 \le i \le d,
\end{equation}
where $\k$ is a non-negative real number and $(i,j)$ denotes the transposition of exchanging $i$th and $j$th elements. 
The Dunkl operators commute in the sense that $D_iD_j = D_jD_i$ for $1\le  i, j \le d$. A linear operator, 
denote by $V_\k$, is called an intertwining operator if it satisfies the relations \cite{D91}
\begin{equation}\label{eq:intertw}
 D_iV_\k =V_\k \partial_i, \qquad 1 \le i \le d. 
\end{equation}
This operator is uniquely determined if it also satisfies $V_\k1 = 1$ and $V_\k: \CP_n^d \mapsto \CP_n^d$, 
where $\CP_n^d$ is the space of homogeneous polynomial of degree $n$ in $d$ variables. 

The commuting property of the Dunkl operators leads to the definition of an analogue of the Laplace operator, 
$\Delta_\k = \sum_{i=1}^d D_i^2$. While the Laplace operator is essential for analysis in $L^2(\RR^d)$, the
operator $\Delta_\k$ plays its role in $L^2(h_\k^2; \RR^d)$, where $h_\k$ is a function invariant 
under the reflection group $G$. In the case of $G = S_d$, the weight function $h_\k$ is defined by 
\begin{equation}\label{eq:hk}
  h_\k(x) = \prod_{1\le i < j \le d} |x_i - x_j|^{\k}, \qquad x \in \RR^d, \quad \k \ge 0.
\end{equation}
In particular, a homogeneous polynomial $Y$ in $d$ variables is called an $h$-harmonic if $\Delta_\k Y =0$. 
The restriction of $h$-harmonics on the unit sphere $\sph$, called spherical $h$-harmonics, are orthogonal.
More precisely, let $\CH_n^d(h_\k^2)$ be the space of $h$-harmonic polynomials of degree exactly $n$. Then 
$h$-harmonics of different degrees are orthogonal: for $Y_n \in \CH_n^d(h_\k^2)$ and $Y_m \in \CH_m^d(h_\k^2)$, 
$$
    \int_{\sph} Y_n(x) Y_m(x) h_\k^2(x) \dd\s(x) = 0, \qquad n \ne m,
$$
where $d\s$ is the surface measure. The theory of spherical $h$-harmonics resembles that of ordinary spherical harmonics. 
In particular,  
$$
   \dim(n,d):= \dim \CH_n^d(h_\k^2) = \binom{n+d-1}{n} - \binom{n+d-3}{n-2}. 
$$
The reproducing kernel $P_n(h_\k; \cdot,\cdot)$ of the space $\CH_n^d(h_\k^2)$ enjoys an addition formula 
given in terms of the intertwining operator $V_\k$. Let $\{Y_{n,\ell}: 1 \le \ell \le \dim(n,d)\}$ be an orthonormal basis 
of $\CH_n^d(h_\k^2)$. Then the kernel  $P_n(h_\k^2)$ satisfies
\begin{equation}\label{eq:Pn-kernel}
   P_n(h_\k^2; x,y) = \sum_{\ell=1}^{\dim(n,d)} Y_{n,\ell}(x) Y_{n,\ell}(y). 
\end{equation}
The addition formula of the kernel is given by \cite{X97} 
\begin{equation}\label{eq:addition}
   P_n(h_\k^2;x,y) = V_\k \left[Z_n^{\l_\k} (\la \cdot ,y\ra) \right](x), \qquad x,y \in \sph,
\end{equation}
where $Z_n^\l$ is given in terms of the classical Gegenbauer polynomial $C_n^\l$ by 
\begin{equation}\label{eq:Zn}
   Z_n^\l(t) = \frac{n+\l}{\l} C_n^\l(t), \qquad -1 \le t \le 1, 
\end{equation}
and $\l_\k$ is a constant that is given by, when $G = S_d$, 
\begin{equation}\label{eq:lambdak}
     \l_\k:= \binom{d}{2} \k + \frac{d-2}{2}.
\end{equation}
When $\k=0$, $V_\k = id$ and the identity \eqref{eq:addition} coincides with the addition formula of 
ordinary spherical harmonics. 

The reproducing kernel $P_n(h_\k^2)$ is the kernel of the orthogonal projection operator 
$\proj_n^\k: L^2(h_\k^2,\sph) \mapsto \CH_n^d(h_\k^2)$ and it plays a central role in the study of 
Fourier orthogonal series in spherical $h$-harmonics, which we shall call spherical $h$-harmonic series
from now on. For intrinsic properties that rely on the underlying reflection group of such series, we need 
a closed formula for the kernel, which calls for an {\it explicit} integral representation of $V_\k$. It is known 
\cite{R} that there exists a nonnegative probability measure $d\mu_x$ such that 
$V_\k f (x) = \int_{\RR^d} f(y) d\mu_x(y)$. What we need, however, is a far more explicit representation.
At the moment, such a representation is known for $G= \ZZ_2^d$ with $h_\k(x ) = \prod_{i=1}^d |x_i|^{\k_i}$, 
which allows us to carry out hard analysis and establish several fundamental results on the spherical 
$h$-harmonic series; see \cite{DaiX}.  For the symmetric group $S_3$, a version of the integral representation
was obtained in \cite{D95}, which however is not adequate for hard analysis of the spherical $h$-harmonic 
series. It should be mentioned that an integral representation of the generalized spherical functions associate 
to $S_d$ was given recently in \cite{Sa}, which is closely related to the intertwining operator. 

Our main result of the present paper is to provide a explicit integral representation for $V_\k f$ associate 
to $S_d$ when the function $f$ depends only on one component of its variables. The integral is over a 
regular simplex in $d-1$ variables and is motivated by our recent work \cite{X19}, where such an integral 
is used for a representation of $V_\k$ for the dihedral group. As an application, we obtain a closed formula 
for the reproducing kernel $P_n(h_\k^2; x, e_j)$, where $e_j$ is the $j$-th coordinate vector, which allows 
us to study the $h$-harmonic series at $e_j$. By taking an integral average over $\sph$, the Ces\`aro 
$(C,\d)$ means of $h$-harmonic series are known \cite{X97} to converge if $\d  > \l_\k$ in $L^1(h_\k^2;\sph)$
or in $C(\sph)$, but the result is not sharp since taking average over sphere removes the action of the group 
inadvertently. Using the new integral representation of $V_\k$, we shall show that the convergence holds 
if $\d > \l_\k$ is replaced by $\d > \l_\k - (d-1)\k$. 

The paper is organized as follows. The new integral representation will be stated and proved in the next 
section, where several of its consequences will also be discussed. The spherical $h$-harmonic series is 
considered in Section 3, where the convergence of the $(C,\delta)$ means at coordinate vectors is 
established, assuming a critical estimate over an integral of the Jacobi polynomial. The latter estimate is 
technical and will be carried out in Section 4. 

\section{Intertwining operator associated to symmetric groups}
\setcounter{equation}{0}

Let $V_\k$ be the intertwining operator associated to the symmetric group $S_d$. Our main result 
in this section is the following integral representation of $V_\k$. Let $T^{d-1}$ denote the simplex 
$$
T^{d-1}:= \big\{u \in \RR^{d-1} : t_1 \ge 0, \ldots, t_{d-1} \ge 0, \, t_1+\cdots + t_{d-1} \le 1\}. 
$$
Written in homogeneous coordinates of $\RR^d$, it is equivalent to the simplex 
$$
    \CT^d =\big\{(t_0,\ldots,t_{d-1}) \in \RR^d: t_i \ge 0, \quad t_0 + t_1+\cdots + t_{d-1} =1\big\}.  
$$

\begin{thm} \label{thm:VkSym}
Let $f: \RR \to \RR$. For $1 \le \ell \le d$, define $F(x_1,\ldots, x_d) = f(x_\ell)$. Let 
\begin{equation}\label{eq:Vk}
  V_\k F(x) = c_\k \int_{\CT^d} f(x_1 t_0 + x_2 t_1+ \cdots + x_d t_{d-1}) t_{\ell-1} (t_0 \ldots t_{d-1})^{\k-1} \dd t,
\end{equation}
where the constant $c_\k$ is given by
$$
c_\k = c_{\k,d} =\Gamma(d \k + 1)/(\k \Gamma(\k)^d).
$$ 
Then the operator $V_\k$ satisfies 
$$
  D_i V_\k F(x) =  V_\k (\partial_i F)(x), \qquad 1 \le i \le d.
$$
\end{thm}

\begin{proof}
The constant $c_\k$ is chosen so that $V_\k 1 = 1$. By the symmetry in the formula of \eqref{eq:Vk}, 
it is sufficient to consider $\ell =1$. Let $F(x) = f(x_1)$. Exchanging the variables $t_0$ and $t_{j-1}$ 
in the integral, we see that 
$$
V_\k F (x (1j)) = c_\k \int_{\CT^d} f \left(x_1 t_0 + \cdots + x_d t_{d-1} \right) t_{j-1} 
      (t_0 \cdots t_{d-1})^{\k-1} \dd t,
$$
which leads immediately to, setting  $t_0 =1-t_1-\cdots - t_{d-1}$, 
$$  
   V_\k F(x) - V_\k F (x (1j)) = c_\k \int_{T^{d-1}} f \left(x_1 t_0 + \cdots + x_d t_{d-1} \right) (t_0-t_{j-1}) 
      (t_0 \cdots t_{d-1})^{\k-1} \dd t.  
$$
Since $\k (t_0-t_j) t_0^{\k-1} t_j^{\k-1}  = \frac{\partial}{\partial t_j} (t_0 t_j)^\k$ for $j \ge 1$, 
integration by parts gives 
\begin{align*}
   \k \frac{  V_\k F(x) - V_\k F (x (1j)) } {x_1-x_j} = & \,
        c_\k  \int_{T^{d-1}}  f' \left(x_1 t_0 + \cdots + x_d t_{d-1} \right)  t_0 t_j (t_0 \cdots t_{d-1})^{\k-1} \dd t. 
\end{align*}
Moreover, taking derivative shows 
\begin{align*}
 \frac{\partial}{\partial x_1} V_\k F(x) = &\, 
    c_\k  \int_{T^{d-1}}  f' \left(x_1 t_0 + \cdots + x_d t_{d-1} \right)  t_0^2 (t_0 \cdots t_{d-1})^{\k-1} \dd t.
 \end{align*}
Hence, adding the terms together, we obtain
\begin{align*}
 D_1 V_\k F_\k (x) = 
   c_\k  \int_{T^{d-1}}  & f' \left(x_1 t_0 + \cdots + x_d t_{d-1} \right) (t_0^2+ t_0 (t_1+\ldots t_{d-1})) \\
      & \times  (t_0 \cdots t_{d-1})^{\k-1} \dd t  = V_\k (\partial_1 F)(x)
\end{align*}
upon using $t_0 +t_1 + \cdots + t_{d-1} =1$. Furthermore, since $t_0=1-t_1-\cdots - t_{d-1}$ is symmetric in 
$t_1,\ldots, t_{d-1}$, it is easy to see that $V_\k F(x) - V_\k F (x (i,j)) =0$ for $2 \le i,j \le d$.  Moreover, 
for $i \ge 2$, 
\begin{align*}
\frac{\partial}{\partial x_i} V_\k F(x) = &\, 
   c_\k  \int_{T^{d-1}}  f' \left(x_1 t_0 + \cdots + x_d t_{d-1} \right) t_0 t_{i-1} (t_0 \cdots t_{d-1})^{\k-1} \dd t.
\end{align*}
Hence, it follows that, for $i =2,3,\ldots,d$,  
\begin{align*}
 D_i V_\k F(x)&\,  = \frac{\partial}{\partial x_i} V_\k F(x) + \k \frac{V_\k F(x) - V_\k F(x (i, 1 ))} {x_i - x_1} \\
  &\,  =  c_\k \int_{T^{d-1}} f' \left(x_1 t_0 + \cdots + x_d t_{d-1} \right)
      ( t_0 t_{i-1}  - t_0 t_{i-1})(t_0 \cdots t_{d-1})^{\k-1}    \dd t \\
    &  =0 = V_\k \partial_i F (x).
\end{align*}
Putting these together, we have complete the proof. 
\end{proof}

The integral over the simplex $\CT^d$ is also used in \cite{X19} for an integral representation of the 
intertwining operator associated to the dihedral group of $d$-regular polygon in $\RR^2$, and the integral 
representation is also given for functions that depend only on one variable. 

Although \eqref{eq:Vk} is suggestive, we do not have an integral expression for a generic function 
$f: \RR^d \to \RR$ for $d > 2$. In the case $d=2$, \eqref{eq:hk} becomes  
$$
   h_\k(x_1,x_2) = |x_1-x_2|^\k,
$$
for which we can deduce a complete integral representation for $V_\k$. This formula, stated below, can 
also be deduced from the formula for the weight function $h_{\l,\mu} (x) = |x_1|^\l |x_2|^\mu$, associated 
to the dihedral group $I_2$, by a rotation of $90^\circ$ and setting $\l = \k$ and $\mu = 0$. For the record,  
we give a proof that verifies it directly from the definition. 

\begin{thm} \label{prop:Vd=2}
For the group $S_2$ and in homogeneous coordinates $t_0 + t_1 =1$, 
\begin{equation} \label{eq:Vd=2}
  V_\k f(x_1,x_2) = c_\k \int_{\CT^2} f \left(x_1 t_0  +x_2 t_1,  x_1t_1+ x_2 t_0 \right) t_0^\k t_1^{\k-1} \dd t. 
\end{equation}
\end{thm}

\begin{proof}
We verify the righthand side of \eqref{eq:Vd=2} satisfies the definition of $V_\k$. First, 
$$  
   V_\k f(x_1,x_2) - V f (x_2,x_1) = c_\k \int_{\CT^2} f \left(x_1 t_0  +x_2 t_1,  x_1t_1+ x_2 t_0 \right) 
      (t_0-t_1) t_0^{\k-1} t_1^{\k-1} \dd t.
$$
Since $\k (t_0-t_1) t_0^{\k-1} t_1^{\k-1}  = \frac{\partial}{\partial t_1} (t_0 t_1)^\k$, integration by parts gives 
\begin{align*}
   \k \frac{ V f(x_1,x_2) - V f (x_2,x_1)}{x_1-x_2} = & \,
      c_\k  \int_{\CT^2}  \partial_1 f \left(x_1 t_0  +x_2 t_1,  x_1t_1+ x_2 t_0 \right)  t_0^{\k} t_1^{\k} \dd t \\
     -  &\,  c_\k  \int_{\CT^2}  \partial_2 f \left(x_1 t_0  +x_2 t_1,  x_1t_1+ x_2 t_0 \right)  t_0^{\k} t_1^{\k} \dd t. 
\end{align*}
Taking derivative gives 
\begin{align*}
 \frac{\partial}{\partial x_1} V f(x_1,x_2) = &\, 
   c_\k  \int_{\CT^2}  \partial_1 f \left(x_1 t_0  +x_2 t_1,  x_1t_1+ x_2 t_0 \right)  t_0^{\k+1} t_1^{\k-1} \dd t \\
     +  & \, c_\k  \int_{\CT^2}  \partial_2 f \left(x_1 t_0  +x_2 t_1,  x_1t_1+ x_2 t_0 \right)  t_0^{\k} t_1^{\k} \dd t. 
\end{align*}
Hence, adding the two terms together, we obtain
\begin{align*}
 D_1 V f(x_1,x_2) =  &\, 
   c_\k  \int_{\CT^2}  \partial_1 f \left(x_1 t_0  +x_2 t_1,  x_1t_1+ x_2 t_0 \right)  t_0^{\k} t_1^{\k-1} \dd t = V(\partial_1 f)(x),
\end{align*}
where we have used $t_0^{\k} t_1^{\k} +  t_0^{\k+1} t_1^{\k-1} = t_0^{\k} t_1^{\k-1}$, which follows from
 $t_0 +t_1 =1$.  The same consideration works for $D_2 V = V\partial_2$. Notice that the denominator of the
 difference operator for $D_2$ is $x_2-x_1$ instead of $x_1-x_2$. 
\end{proof}

Let $\CH_n^d(h_\k^2)$ be the space of spherical $h$-harmonics of degree $n$. We denote by   
$e_1:=(1,0,\cdots, 0), \ldots, e_d:= (0,\ldots, 0,1)$ the standard coordinate vectors of $\RR^d$. 

\begin{prop}
For $1 \le \ell \le d$, the reproducing kernel $P_n(h_\k; \cdot,\cdot)$ of $\CH_n^d(h_\k^2)$ satisfies 
\begin{equation}\label{eq:Pn(x,e)}
  P_n(h_\k^2; x, e_\ell) = c_\k \int_{\CT^d} Z_n^{\l_\k} (x_1t_0 + \cdots + x_d t_{d-1}) t_\ell (t_0\cdots t_{d-1})^{\k-1} \dd t.
\end{equation}
\end{prop}

This follows immediately from the addition formula \eqref{eq:addition} and the integral representation 
\eqref{eq:Vk}. The identity \eqref{eq:Pn(x,e)} plays an essential role in our study in the next section. 

Another important extension from $L^2(\RR^d)$ to $L^2(\RR^d, h_\k^2)$ is an analogue of the Fourier 
transform in the Dunkl setting. For the symmetric group $S_d$, this transform is defined by \cite{D91}
$$
   \CF_\k f(x) =  c_\k' \int_{\RR^d} f(y) E_\k (x, \ii y) h_\k^2(y) \dd y, \qquad c_\k' = 
        \frac{\Gamma(\f{d}2)} {(2 \pi)^{\f{d}{2}}2^{\binom{d}{2}\k}\Gamma\big(\binom{d}{2}\k+\f{d}{2}\big)} c_\k, 
$$
where the exponential function $E_\k$ is defined by
$$
  E_\k (x,y) := V_\k \left[ \ee^{\la x,\cdot\ra} \right](y), \qquad x,y \in \CC^d. 
$$
It is known that $E_\k(x,y) := E_\k(y,x)$. Furthermore, the generalized Bessel function $K_\k = K_{\k,d}$ is defined by
$$
  K_{\k,d}(x,y) = \frac{1}{d!} \sum_{\s \in S_d} E_\k(x, y \sigma). 
$$

For the symmetric group $S_d$, the formula in Theorem \ref{thm:VkSym} gives the following:

\begin{cor}
For $\tb = (t_0,\ldots, t_{d-1}) \in \CT^d$ and $1 \le \ell \le d$,  
$$
  E_\k(e_\ell, y) =  c_\k \int_{\CT^d} \ee^{\la y, \tb\ra } t_{\ell-1} (t_0 \ldots t_{d-1})^{\k-1} \dd t.
$$
Furthermore, the generalized Bessel function satisfies $K_{\k,d}(e_\ell, y) = K_{\k,d} (e_1,y)$ and 
$$
  K_{\k,d}(e_1,y) =  \frac{1}{d} \sum_{j=1}^d E_\k (e_1, y(1,j))= \frac{c_\k}{d} \int_{\CT^d} \ee^{\la y, \tb\ra } (t_0 \ldots t_{d-1})^{\k-1} \dd t.
$$
\end{cor}

\begin{proof}
The symmetric group $S_d$ can be decomposed as the left cosets of $S_{d-1}$ given by $S_{d-1}(1,j)$, 
$1 \le j\le d$, which proves the first identity in $K_{\k,d}(e_1,y)$. From the expression of $E_\k(e_\ell,y)$, 
it is easy to see that $E_\k(e_1, y(1,j)) = E_\k(e_j, y)$, so that the second expression for $K_{\k,d}$ follows 
readily from $t_0+\cdots + t_{d-1} =1$. 
\end{proof}

For $d =2$, we can use Proposition \ref{prop:Vd=2} to write $K_{\k,2}(x,\ii y)$ in terms of the classical Bessel
function $J_\nu$, which satisfies the integral formula
$$
  J_\nu(z) = \frac{(\frac{z}{2})^\nu}{\pi^{\f12}\Gamma(\nu+\f12)} \int_{-1}^1 \ee^{\ii z t}(1-t^2)^{\nu - \f12} \dd t.
$$
 
\begin{thm}
For $x, y \in \RR^2$, 
\begin{align} \label{eq:Besseld=2}
  K_{\k,2}(x, \ii\, y)  = &\,    \sqrt{\pi} \Gamma\left(\k+\tfrac12\right)  \ee^{\ii \frac{(x_1 + x_2) (y_1 + y_2)}{2}}\\
         & \times  \bigg(\frac{2} {(x_1 -x_2)(y_1- y_2)}\bigg)^{\k-\f12} J_{k-\f12}\left(\frac{(x_1 -x_2)(y_1- y_2)}{2}\right). \notag
\end{align}
Furthermore, for $d > 2$, the generalized Bessel function satisfies 
\begin{align} \label{eq:Besseld>2}
   K_{\k,d}(e_1, \ii \,y) = \frac{c_{\k,d}}{c_{\k,d-1}} \int_0^1 \ee^{\ii r y_d} K_{\k,d-1}\big(e_1, \ii (1-r) y'\big) 
         r^{\k-1} (1-r)^{(d-1)\k-1} \dd r,
\end{align}
where $y = (y', y_d) \in \RR^d$ and $e_1 = (1,0,\ldots 0)$ in either $\RR^d$ or $\RR^{d-1}$.
\end{thm}

\begin{proof}
Using \eqref{eq:Vd=2} with $t_0 = 1-t$ and $t_1=t$, and then changing variable $t \mapsto (1+s)/2$, we obtain
\begin{align*}
  K_{\k,2}(x, \ii \,y)  &\, =  \frac{\Gamma(2\k+1)}{2\Gamma(\k)\Gamma(\k+1)} 
       \int_0^1 \ee^{\ii (\la x,y\ra - (x_1 -x_2)(y_1- y_2)t)} (1-t)^{\k-1}t^{\k-1} \dd t \\
    &\, =  \ee^{\ii \frac{(x_1 + x_2) (y_1 + y_2)}{2}}\frac{\Gamma(\k+\f12)}{\Gamma(\k)}
        \int_{-1}^1 \ee^{- \ii \frac{(x_1 -x_2)(y_1- y_2)}{2} s} (1-s^2)^{\k-1} \dd s,
\end{align*}
where the constant has been simplified using the formula for $\Gamma(2 a)$. Writing the last integral in 
terms of $J_{\k-\f12}$ proves \eqref{eq:Besseld=2}. 

For $d > 2$, we denote by $\CT_\rho^d$ the simplex $\{t \in \RR_+^d:  t_0 + t_1 +\cdots t_{d-1} = \rho\}$. Then
\begin{align*}
  \int_{\CT^d} g(u) \dd u & = \int_0^1 \int_{\CT^{d-1}_{1-t_{d-1}}} g\big(t_0,\ldots, t_{d-2}, t_{d-1}\big) \dd t \\
        & = \int_0^1 (1-r)^{d-2} \int_{\CT^{d-1}} g \big( (1-r) s, r\big) \dd s \dd r, 
\end{align*}
where we set $t_{d-1} = r$ and $t_i = (1-r) s_i$ for $i = 1,2,\ldots, d-2$, which also implies that $t_0 = (1-r)s_0$. 
Setting $g(t) = \ee^{\ii (t_0 y_1 + \ldots + t_{d-1} y_d)}$ in the above identity, the recursive 
formula \eqref{eq:Besseld>2} follows readily. 
\end{proof}
 
For $x, y$ in the domain $\{x \in \RR^d: x_1 + x_2 + \cdots x_d =0\}$, a fairly involved recursive formula for the 
generalized Bessel functions associated to the symmetric group $S_d$, or root system $A_{d-1}$, is given in 
\cite{Am}. The domain, however, does not contain coordinate vectors $e_\ell$. In the case of $d =2$, it agrees 
with \eqref{eq:Besseld=2} with $x_2 = - x_1$ and $y_2 = - y_1$ apart from an extra constant $\sqrt{\pi}$. 

Finally, let us mention a property of the intertwining operator $V_\k$ that we shall need in the next section. 
We denote by $a_\k$ the normalization constant of $h_\k^2$ defined by $a_\k \int_{\sph} h_\k^2(x) d\s(x) =1$. 
For the symmetric group $S_d$, we have \cite[p. 216 and Thm 10.6.17]{DX} 
$$
  a_\k = \frac{2^{ \binom{d}{2}\k}}{\o_d} \frac{\Gamma( \binom{d}{2}\k+\frac{d}{2})}{\Gamma(\f{d}{2})} \prod_{j=2}^d \frac{\Gamma(\k+1)}{\Gamma(j\k+1)},
$$
where $\o_d =  2 \pi^{\f{d}{2}}/\Gamma(\frac{d}{2})$ is the surface area of $\sph$. 

\begin{lem}\label{lem:intV}
Let $f: \RR\mapsto \RR$ be a function such that both integrals below are defined. Then for $x \in \RR^d$, 
\begin{equation}\label{eq:intV}
   a_\k \int_{\sph} V_\k [f(\la x,\cdot\ra)](y) h_\k^2(y) \dd\s(y) = b_{\l_\k} \int_{-1}^1 f(\|x\| t) (1-t^2)^{\l_\k -\f12} \dd t,
\end{equation}
where $b_\l$ is the constant so that $b_\l \int_{-1}^1(1-t^2)^{\l-\f12}\dd t =1$. 
\end{lem}

A more general identity holds for $V_\k f$ for generic function $f: \RR^d \mapsto \RR$, where the righthand side 
is replaced by the integral of $f(x)$ with respect to $(1-\|x\|^2)^{\g_\k -1}$ over the unit ball of $\RR^d$ \cite{X97}.
The identity shows that taking the average over the sphere removes the action of the reflection group.

\section{spherical $h$-harmonic  series}
\setcounter{equation}{0}

In the first subsection we outline the background and what is known for the spherical $h$-harmonic  series
in the setting of a generic reflection group. The new result for the symmetric group is given in the second
subsection.

\subsection{Spherical $h$-harmonic  series}
Let $G$ be a given reflection group. Let $h_\k$ be the $G$-invariant function with respect to which that
spherical $h$-harmonics are orthogonal. When $G$ is the symmetric group $S_d$, the function $h_\k$ is 
given in \eqref{eq:hk}. Another case of interests for our discussion is $G = \ZZ_2^d$, the group of sign
changes, that has 
\begin{equation} \label{eq:hkZ2d}
   h_\k(x) = \prod_{i=1}^d |x_i|^{\k_i}, \qquad \k_i \ge 0, \quad 1 \le i \le d.
\end{equation}
Unless specified otherwise, the discussion in this subsection holds for spherical $h$-harmonics associated
with a generic reflection group $G$; see, for example, \cite{DX}.

For $n \ge 0$, let $\{Y_{n,\ell}: 1 \le \ell \le \dim(n,d)\}$ be an orthonormal basis of $\CH_n^d(h_\k^2)$, 
normalized with respect to the inner product
$$
      \la f,g\ra_\k = a_\k \int_{\sph} f(x) g(x) h_\k^2(x) d\s(x).
$$
For $f \in L^2(\sph, h_\k^2)$, the spherical $h$-harmonic series of $f$ is defined by 
$$
  L^2(h_\k^2) = \bigoplus_{n=0}^\infty \CH_n^d(h_\k^2): \qquad f  = \sum_{n=0}^\infty \sum_{\ell =1}^{\dim(n,d)} \la f, Y_{n,\ell} \ra_\k Y_{n,\ell}. 
$$
The projection operator $\proj_n^\k: L^2(h_\k^2) \mapsto \CH_n^d(h_\k^2)$ and the $n$-th partial sum operator 
$S_n(h_\k^2;f)$ are defined by 
$$
\proj_n^\k f =  \sum_{\ell =1}^{\dim(n,d)} \la f, Y_{n,\ell} \ra_\k Y_{n,\ell} \quad\hbox{and}\quad
 S_n(h_\k^2;f) = \sum_{m=0}^n \proj_m^\k f.
$$ 
By the definition of the reproducing kernel $P_n(h_\k^2)$ in \eqref{eq:Pn-kernel}, it follows readily that
\begin{align} \label{eq:projZn}
   \proj_n^\k f(x) & = a_\k \int_{\sph} f(y) P_n(h_\k^2;x,y) h_\k^2(y) d\s(y) \\
               & = a_\k \int_{\sph} f(y) V_\k \left[Z_n^{\l_\k} (\la x, \cdot \ra)\right](y) h_\k^2(y) d\s(y), \notag
\end{align}
where the second identity follows from the addition formula \eqref{eq:addition}. By the definition of $Z_n^\l$ 
in \eqref{eq:Zn}, the partial sum operator is related to that of Fourier series in the Gegenbauer polynomials. 

The $n$-th partial sum operator $S_n(h_\k^2; f)$ converges to $f$ in the $L^2(h_\k^2)$ norm by the classical 
Hilbert space theory. For $p \ne 2$, we consider the convergence of the Ces\`aro means, which often
serve as a test stone of our knowledge on summability method. For $\delta > -1$, the 
Ces\`aro $(C,\delta)$ means of the spherical $h$-harmonic  series are defined by 
\begin{align} \label{eq:Sdelta}
 S_n^\d (h_\k^2;f) : = &\, \frac{1}{ \binom{n+\delta}{n}}\sum_{k=0}^n   \binom{n-k+\delta}{n-k} \proj_n^\k f \\
   = &\, a_\k \int_{\sph} f(y) K_n^\d (h_\k^2;x,y) h_\k^2(y) \dd y,   \notag
\end{align}
where $K_n^\d (h_\k^2; \cdot,\cdot)$ is the $(C,\d)$ means of $Z_m^{\l_\k}(\la \cdot,\cdot\ra)$ and the second
identity follows from \eqref{eq:projZn}.  
Let $w_\l(t) = (1-t^2)^{\l-\f12}$ be the Gegenbauer weight function. Denote by $k_n^\delta(w_\l; s,t)$ 
the kernel of the $(C,\delta)$ means of the Fourier--Gegenbauer series on $[-1,1]$; see the definition 
in the next section. Then we can derive from \eqref{eq:Zn} that
\begin{equation} \label{eq:SdKernel}
  K_n^\d (h_\k^2;x,y) =  V_\k \left[k_n^\delta (w_{\l_\k}; \la x, \cdot \ra, 1)\right](y). 
\end{equation} 

Since the $(C,\delta)$ means are linear integral operators, we know that $S_n^\delta (h_\k^2; f)$ converges 
in $L^1(h_\k^2; \sph)$ or in $C(\sph)$  if and only if 
\begin{equation} \label{eq:conv-iff}
   \sup_{x \in \sph} \int_{\sph} \left| K_n^\delta(h_\k^2; x, y)\right| h_\k^2(y) d\s(y) < \infty.
\end{equation} 
For $1 \le p \le \infty$, let $\|\cdot\|_{p,\k}$ denote the $L^p(h_\k^2, \sph)$ norm for $1 \le p < \infty$ and 
the uniform norm of $C(\sph)$ for $p = \infty$. A sufficient condition for the convergence of the $(C,\delta)$ 
means of spherical $h$-harmonics was given in \cite{X97}. 

\begin{thm} \label{thm:conv-Snd}
Let $f \in L^p(h_\k^2, \sph)$, $1\le p < \infty$ or $f \in C(\sph)$. Then the $(C,\delta)$ means $S_n^\d (h_\k^2;f)$ 
converges to $f$ in $\|\cdot\|_{\k,p}$ norm if $\d > \l_\k$. 
\end{thm}

The proof follows from Lemma \ref{lem:intV} and \eqref{eq:SdKernel}, which reduces the boundedness 
in \eqref{eq:conv-iff} to the boundedness of $\int_{-1}^1 |k_n^\delta (w_{\l_\k}; \la x, \cdot \ra, 1)| w_{\l_\k}(t) \dd t$, 
and the latter holds if and only if $\d > \l_\k$ by the classical result of Szeg\H{o} \cite[Theorem 9.1.3]{Sz}. 
The case $1 < p < \infty$ follows from the Riesz interpolation. 

The theorem holds for spherical $h$-harmonics  series associated with all reflection groups. The use of 
\eqref{eq:intV}, however, removes the action of reflection group altogether and, as a consequence and not
surprisingly, we pay the price that the condition $\d > \l_\k$ is not sharp in general. This is first illustrated 
in the case when $G= \ZZ_2^d$. For the weight function $h_\k(x) = \prod_{i=1}^d |x_i|^{\k_i}$ given 
in \eqref{eq:hkZ2d}, the intertwining operator $V_\k$ has an integral representation
$$
 V_\k f(x) = c_\k \int_{[-1,1]^d} f(x_1t_1,\ldots,x_d t_d) \prod_{i=1}^d (1+t_i)(1-t_i^2)^{\k_i -1} \dd t.
$$
This leads to, by \eqref{eq:SdKernel}, a closed formula for the kernel $K_n^\d(h_\k^2;\cdot,\cdot)$, which
makes it possible to obtain a sharp estimate of the kernel that can be used to determine the critical index of 
the $(C,\delta)$ means. Indeed, while Theorem \ref{thm:conv-Snd} establishes the convergence for 
$\delta > \l_\k = \sum_{i=1}^d \k_i + \f{d-2}{2}$ in this setting, it was proved in \cite{LX} that the 
Ces\`aro means $S_n^\d (h_\k^2;f)$ converges to $f$ in $\|\cdot\|_{\k,p}$ norm for $p=1$ or $\infty$ if 
and only if $\d > \l_\k - \min\limits_{1\le i\le d} \k_i$ for $G= \ZZ_2^d$.

We shall show in the next subsection that our new integral representation of $V_\k$ for the symmetric group
will allow us to establish a similar result for the symmetric group, albeit only for convergence of $S_n^\d (h_\k^2;f)$ 
at coordinate vectors. 

\subsection{Spherical $h$-harmonics series associated to symmetric group}
In this subsection, $G$ is the symmetric group $S_d$ and $h_\k$ is given by \eqref{eq:hk}. Recall that
$\l_\k = \binom{d}{2}\k + \frac{d-2}{2}$ by \eqref{eq:lambdak}. Our main result is the following theorem on 
the Ces\`aro $(C,\delta)$ means of spherical $h$-harmonic series. 

\begin{thm} \label{thm:cesaro-sd}
Let $h_\k$ be defined as in \eqref{eq:hk} and $f \in C(\sph)$. Then $S_n^\d (h_\k^2, f)$ converges to
$f$ at the coordinate vectors $e_\ell$, $1 \le \ell\le d$, if 
\begin{equation}\label{eq:cesaro-sd}
\d > \lambda_\k - (d-1)\k = \binom{d-1}{2} \k + \f{d-2}{2}. 
\end{equation}
\end{thm}

The proof requires a sharp estimate of the kernel $K_n^\d(h_\k^2; x,e_\ell)$, which comes down to estimate 
an integral of Jacobi polynomials over the simplex. We start by recalling the Jacobi polynomials and the 
definition of the kernel $k_n^\d$. 

The Jacobi polynomials $P_n^{(\a,\b)}$ are orthogonal with respect to the weight function
$w_{\a,\b}(t) = (1-t)^\a (1+t)^\b$ on $[-1,1]$. For $g \in L^1(w_{\a,\b}; [-1,1])$, let $s_n^\delta (w_{\a,\b};g)$ 
denote the Ces\`aro $(C,\delta)$ means of the Fourier-Jacobi series. Then 
$$
  s_n^\delta (w_{\a,\b};g) = c_{\a,\b} \int_{-1}^1 g(s) k_n^\delta (w_{\a,\b}; \cdot ,s) w_{\a,\b}(s) \dd s,
$$  
where the kernel $k_n^\delta (w_{\a,\b}; \cdot, \cdot)$ is given by 
$$
  k_n^\delta (w_{\a,\b}; s,t) = \frac{1}{\binom{n+\delta}{n}} \sum_{k=0}^n \binom{n-k+\delta}{n-k} 
        \frac{P_k^{(\a,\b)}(s)P_k^{(\a,\b)}(t)}{h_k^{(\a,\b)}},
$$
in which $h_k^{(\a,\b)}$ is the $L^2(w_{\a,\b}, [-1,1])$ norm of $P_n^{(\a,\b)}$. The Gegenbauer polynomials
are related to the Jacobi polynomials by 
$$
  C_n^\l(t) = \frac{(2\l)_n}{(\l+\f12)_n} P_n^{(\l-\f12,\l-12)}(t) 
$$
and they are orthogonal with respect to $w_\l(t) = w_{\l-\f12,\l-\f12}(t)$ on $[-1,1]$. In particular, 
$k_n^\d(w_{\l}; \cdot,\cdot)= k_n^\d(w_{\l-\f12,\l-\f12}; \cdot,\cdot)$. 

Throughout the rest of this paper, we denote by $c$ a generic constant that may depend on fixed parameters
such as $\k$ and $d$, and its value may change from line to line. Furthermore, we write $A\sim B$ if 
$A \le c B$ and $B \le c A$. 
  
Our starting point is the following result in \cite[p. 261, (9.4.13)]{Sz}, which shows that the main term of 
$k_n^\d(w_{\a,\b};\cdot,1)$ is a Jacobi polynomial. 

\begin{lem} \label{lem:knd-sum}
For any $\alpha,\beta > -1$ such that $\alpha+ \beta+ \delta + 3 >0$,
\begin{equation}\label{eq:knd-sum}
  k_n^\delta(w_{\a,\b}, t,1) = \sum_{j=0}^J b_j(\alpha,\beta,\delta,n)
    P_n^{(\alpha+\delta+j+1,\beta)}(t) + G_n^{\delta}(t),
\end{equation}
where $J$ is a fixed integer and
$$
   G_n^\delta(t) = \sum_{j=J+1}^\infty d_j(\alpha,\beta,\delta,n) k_n^{\delta+j}(w_{\alpha,\beta}, 1,t);
$$
moreover, the coefficients satisfy the inequalities,
$$
 |b_j(\alpha,\beta,\delta,n)| \le c n^{\alpha+1-\delta - j} \quad \hbox{and}
 \quad  |d_j(\alpha,\beta,\delta,n)| \le c j^{-\alpha-\beta-\delta - 4}.
$$
\end{lem}

When $\d$ is large, the kernel $k_n^\d(w_{\a,\b})$ is non-negative and satisfies an estimate given 
in the lemma below, which was first used in \cite{BC} and \cite{CTW}. 

\begin{lem} \label{lem:knd-est}
Let $\alpha,\beta \ge -1/2$.  If $\delta \ge \alpha + \beta + 2$, then
$$
 0 \le k_n^\delta(w_{\alpha,\beta}, t,1)  \le c n^{-1} (1-t+n^{-2})^{-( \alpha+3/2)}.
$$
\end{lem}

We shall use \eqref{eq:knd-sum} to write $K_n^\d(h_\k^2)$ as two terms. For the second term, we 
use the above lemma to estimate the $G_n^\d$ term, which is relatively easy to handle. The main 
effort in estimating the first term lies in the proof of the following theorem. 

\begin{thm} \label{thm:estimate}
Let $\k > 0$ and let $\varphi$ be a $C^\infty$ function on $\CT^d$. If $\a \ge \b$ and $\a \ge (d-1)\k - \f12$,
then  
\begin{align}\label{eq:estimate}
\left|\int_{\CT^d} P_n^{(\a,\b)}(x_1 t_0 + x_2 t_1 + \cdots + x_d t_{d-1}) \varphi(t) (t_0 t_1 \cdots t_{d-1})^{k-1} 
   \dd t \right| \\
  \qquad\qquad\qquad 
    \le c n^{-(d-1)\k - \f12} \sum_{i=1}^d  
        \frac{  {\prod_{j=1, j \ne i }^d |x_j - x_i|^{-\k}}}{\big(\sqrt{1-|x_i|} + n^{-1}\big)^{\a+\f12 - (d-1)\k}}. \notag
\end{align}
\end{thm}
 
The proof of this theorem is technical and will be given in the next section. In the rest of this subsection 
we use this theorem to provide a proof of Theorem \ref{thm:cesaro-sd}, which relies on the following
proposition.  

\begin{prop}
Let $\k > 0$. Then  
\begin{align}\label{eq:estKnd}
 \left | K_n(h_\k^2; x, e_\ell)\right| \le & \, c n^{\l_\k - (d-1)\k - \d} \sum_{i=1}^d  
        \frac{  {\prod_{j=1, j \ne i }^d |x_j - x_i|^{-\k}}}{\big(\sqrt{1-|x_i|} + n^{-1}\big)^{\l_\k - (d-1)\k + \d +1}} \\
           & +  c n^{-1} V_\k \left[ \left (1- \la \cdot, e_\ell \ra + n^{-2}\right )^{-(\l_\k+1)}\right] (x). \notag
\end{align}
\end{prop}

\begin{proof}
By \eqref{eq:SdKernel} and the integral representation of $V_\k$ in \eqref{eq:Vk}, we obtain
$$
  K_n^\d(h_\k^2; x,e_\ell) := c_\k \int_{\CT^d} k_n^\d(w_{\l_\k}; t_0 x_1+ \cdots + t_{d-1} x_d) t_{\ell-1} 
     (t_0 \cdots t_{d-1})^{\k-1} \dd t. 
$$
We replace the kernel $k_n^\d(w_{\l_\k}) = k_n^\d(w_{\l_\k-\f12,\l_\k-\f12})$ by the expansion in 
Lemma \ref{lem:knd-sum}. Let $J = \lfloor 2 \l_\k+1\rfloor$. Then 
\begin{equation*}
K_n^\d(h_\k^2; x,e_\ell) = \sum_{j=0}^J b_j\!\left(\l_\k-\tfrac12, \l_\k-\tfrac12, \delta,n\right) \Omega_j(x) + \Omega(x),
\end{equation*}
where
\begin{equation*} 
\Omega_j (x) =  c_\k \int_{\CT^d} P_n^{(\l_\k+\d + j + \f12,\l_\k-\f12)}(t_0 x_1+ \cdots + t_{d-1} x_d) t_{\ell-1} 
     (t_0 \cdots t_{d-1})^{\k-1} \dd t  
\end{equation*}
and
\begin{equation*} \label{eq:3.6}
\Omega(x) =  c_\k \int_{\CT^d} G_n^\delta (t_0 x_1+ \cdots + t_{d-1} x_d) t_{\ell-1} 
     (t_0 \cdots t_{d-1})^{\k-1} \dd t.
\end{equation*}
For $\Omega_j$, we apply the estimate \eqref{eq:estimate} with $\varphi(t) = t_{\ell-1}$, $\a = \l_\k+\d + j + \f12$
and $\b=\l_\k-\f12$. Together with the estimate of $b_j\!\left(\l_\k-\tfrac12, \l_\k-\tfrac12, \delta,n\right)$, we obtain
that $|b_j\!\left(\l_\k-\tfrac12, \l_\k-\tfrac12, \delta,n\right) \Omega_j(x)|$ is bounded by the first term in the righthand
side of \eqref{eq:estKnd}, hence, so is the sum of these terms over $0\le j \le J$. Our choice of $J$ allows us to use 
the estimate in Lemma \ref{lem:knd-est} to obtain  
$$
      \left |G_n^\delta(t)\right | \le  c n^{-1} \left (1-t+n^{-2}\right )^{\l_\k+1},
$$ 
where we have used the fact that $\sum_{j=J+1}^\infty | d_j(\l_\k-\tfrac12, \l_\k-\tfrac12, \delta,n) |$
is bounded. Consequently, by \eqref{eq:Vk}, it follows that $|\Omega(x)|$ is bounded by the second term in the
righthand side of \eqref{eq:estKnd}, This completes the proof. 
\end{proof}

\medskip
\noindent
{\it Proof of Theorem \ref{thm:cesaro-sd}}. 
By \eqref{eq:conv-iff}, we need to prove that 
$$
    I_n: = \int_{\sph} \left | K_n^\d(h_\k^2; x,e_\ell) \right | h_\k^2(x) \dd\s(x) < \infty
$$
if $\d > \l_\k - (d-1)\k$. By \eqref{eq:estKnd}, $I_n$ is bounded by 
\begin{align}\label{eq:I_n}
   I_n \le \sum_{i=1}^d I_{n,i} +  c n^{-1}   \int_{\sph} V_\k \left[ \left (1- \la \cdot, e_\ell \ra + n^{-2}\right )^{-(\l_\k+1)}\right] (x)      
   h_\k^2(x) \dd\s(x),
\end{align}
where $I_{n,i}$ is defined by 
\begin{align*}
 I_{n,i} = c n^{\l_\k - (d-1)\k - \d}  \int_{\sph}
        \frac{  {\prod_{j=1, j \ne i }^d |x_j - x_i|^{-\k}}}{\big(\sqrt{1-|x_i|} + n^{-1}\big)^{\l_\k - (d-1)\k + \d +1}}  h_\k^2(x) \dd\s(x). \end{align*}
For the second term in the righthand side of \eqref{eq:I_n}, we use \eqref{eq:intV} to bounded it by
\begin{align*}
  c n^{-1} \int_{-1}^1 & \frac{(1-t^2)^{\l_\k- \f12}} {(1-t+n^{-2})^{\l_\k+1} }\dd t 
     =  c n^{-1} \int_0^\pi \frac{(\sin \t)^{2 \l_\k}}{(1-\cos \t + n^{-2})^{\l_\k+1}} \dd\t \\
   & \sim n^{-1} \int_0^{\pi/2} \frac{\t^{2 \l_\k}}{( \t^2 + n^{-2})^{\l_\k+1}} \dd\t
     = \int_0^{n \pi/2} \frac{s^{2\l_\k}} {(1+s^2)^{\l_\k+1}} \dd s < \infty.  
\end{align*} 
Next we estimate $I_{n,i}$ for $1\le i \le d$. If $j \ne i$ and $k \ne i$, then, for $x \in \sph$, 
$$
|x_j - x_k|^2 \le 2(x_j^2+x_k^2) \le 2 \sum_{j\ne i} x_j^2 = 2 (1-x_i^2) \le 4 (1-|x_i|),
$$
so that $|x_j - x_k| \le 2 \sqrt{1-|x_i|}$. Hence, by $h_\k^2(x) = \prod_{1\le j < k \le d} |x_j-x_k|^{2\k}$, we obtain
\begin{align*}
  I_{n,i} \le &\, c n^{\l_\k - (d-1)\k - \d}  \int_{\sph} \frac{  {\prod_{j \ne i, k \ne i} |x_j - x_k|^{2\k}}} 
          {\big(\sqrt{1-|x_i|} + n^{-1}\big)^{\l_\k - (d-1)\k + \d +1}}  \dd\s(x)  \\
           \le &\, c n^{\l_\k - (d-1)\k - \d}  \int_{\sph} \frac{1} 
                    {\big(\sqrt{1-|x_i|} + n^{-1}\big)^{\l_\k - (d-1)\k + \d +1 - 2 \binom{d-1}{2} \k}} \dd\s(x).
\end{align*}
Since $\l_\k - (d-1)\k = \binom{d-1}{2} \k + \f{d-2}{2}$,  the last integral can be rewritten to give
\begin{align*}
     I_{n,i} &\le  \,c n^{d-1}  \int_{\sph} \frac{1}{\big(1+ n \sqrt{1-|x_i|} \big)^{\d - \binom{d-1}{2} \k + \frac{d}{2}} } \dd\s(x) \\
         & \sim  c n^{d-1} \int_0^{\pi/2} \frac{(\sin\t)^{d-2}} {(1+n\t)^{\d - \binom{d-1}{2} \k + \frac{d}{2}}} \dd\t,
\end{align*}
where we have used the spherical coordinates by choosing $x_i = \cos \t$. Using $\sin \t \sim \t$ and changing 
variable $t= n\t$, it is easy to see that the last integral is bounded if and only if $\d > \binom{d-1}{2} \k + \f{d-2}{2}$.
This completes the proof. \qed

\medskip

We conjecture that the condition \eqref{eq:cesaro-sd} in Theorem \ref{thm:cesaro-sd} is sharp; that is,
$S_n^\d(h_\k^2; e_\ell)$ does not converge if $\d = \l_\k - (d-1)\k$. More precisely, we expect the inequality 
\begin{align*}
 \int_{\sph} & \left| \int_{\CT^d}   P_n^{(\a,\b)}(x_1 t_0 + x_2 t_1 + \cdots + x_d t_{d-1})t_0(t_0 t_1 \cdots t_{d-1})^{k-1} 
   \dd t \right| h_\k^2(x) \dd \s(x)\\
   & \qquad\qquad\qquad\qquad \qquad\qquad \qquad\qquad \qquad\qquad  
    \ge c n^{-(d-1)\k - \f12} \log n  
\end{align*}
to hold for $\a = \l_\k + \d +\f12$, $\b= \l_\k - \f12$ and $\d = \l_\k - (d-1)\k$, which would prove the sharpness of
the condition by Lemma \ref{lem:knd-sum}. Furthermore, taking the cue from the classical Fourier-Jacobi series
and spherical $h$-harmonic series with $G= \ZZ_2^d$, we expect that the $(C,\delta)$ means $S_n^\d(h_\k^2; f)$, 
with $h_\k$ in \eqref{eq:hk}, converge for the same $\delta$ that ensures the convergence at the coordinates 
vectors. In other words, we conjecture that the condition $\d > \l_\k - (d-1)\k$ is the necessary and sufficient 
condition for the convergence of $S_n^\d(h_\k^2; f)$ in $\|\cdot\|_{p,\k}$ for $p=1$ and $\infty$. 
 
\section{Proof of Theorem \ref{thm:estimate}}
\setcounter{equation}{0}

The proof is based on a lemma established in \cite{DaiX1}. First we need a definition.  

\begin{defn}
Let $n, v\in\mathbb{N}_0$. A function $f:[-1,1] \to \mathbb{R}$ is said to be in class 
$\mathcal {S}^v_n (\mu)$, if there exist functions $F_j$, $j=0,1,\cdots, v$ on $[-1, 1]$ 
such that $F_j^{(j)}(x)=f(x)$, $x\in [-1,1]$, $0\leq j\leq v$, and 
\begin{equation}\label{2-1}
|  F_j(x)| \le c n^{-2j} \left(1+n \sqrt{1-|x|}\right)^{-\mu-\f12+j},\ \ \ x\in [-1,1],\   \    \  j=0,1,\cdots, v.
\end{equation}
\end{defn}

This definition is motivated by the following two properties of the Jacobi polynomials. The first one is
the well-known identity
\begin{equation} \label{D-Jacobi}
    P_n^{(\a+1, \b+1)}(y) = \frac{2}{n+\a+\b+2} \frac{d}{dy} P_{n+1}^{(\a, \b)}(y)
\end{equation}
and the second one is the pointwise estimate of the Jacobi polynomials in the lemma below (\cite[(7.32.5) and (4.1.3)]{Sz}).

\begin{lem} \label{lem:3.2}
For an arbitrary real number $\alpha$ and $t \in [0,1]$,
\begin{equation} \label{Est-Jacobi}
|P_n^{(\alpha,\beta)} (t)| \le c n^{-1/2} (1-t+n^{-2})^{-(\alpha+1/2)/2}.
\end{equation}
The estimate on $[-1,0]$ follows from the fact that $P_n^{(\alpha,\beta)} (t) = (-1)^nP_n^{(\beta, \alpha)} (-t)$.
\end{lem}
 
In particular, these two properties show that $n^{-\a} P_n^{(\a,\b)} \in \CS_n^v(\a)$ for all $v \in \NN_0$. We
can now state our main lemma. 

\begin{lem} \label{lem:ineq-main}
Let $\k > 0$. For a fixed $b$ with $0 < b <1$, let $\xi\in C^\infty [0,1]$ be such that $\text{supp}\ \xi \subset [0, b]$. 
Let $f\in \mathcal {S}^v_n (\mu)$ with $v \ge |\mu| + 2\k +\f32$. Assume $|x| \le 1$ and $|x+ a t| \le 1$ for $t\in [0,1]$.
 Then,  
\begin{align}\label{eq:ineq-main}
 \left | \int_{0}^1 f(x+ a t) t^{\k-1} \xi(t)\, dt\right |
   \le  c n^{-2\k} |a|^{-\k} \left(1+n \sqrt{1-|x|}\right)^{-\mu-\f12+\k}.
\end{align}
\end{lem}

\begin{proof}
Changing variable $t \mapsto 1-t$, the integral becomes
$$
\int_{0}^1 f(x+ a t) t^{\k-1} \xi(t)\, \dd t = \int_{0}^1 f(x+ a - a t) (1-t)^{\k-1} \xi(1-t)\, \dd t.
$$
We can then apply Lemma 3.4 of \cite{DaiX1} on the integral in the righthand side. 
\end{proof}

The statement of \cite[Lemma 3.4]{DX} is slightly more complicated, with an assumption that $|x| \le 1-|a|$, 
but a close look at the proof shows that it suffices to assume that $|x+ a t| \le 1$ for $t\in [0,1]$. This lemma is
used to prove the following estimate: 

\begin{lem} \label{lem:main}
Let $\k > 0$ and let $\xi$ be a $C^\infty(\CT^d)$ function such that its support set is 
$\{t \in \CT^d: t_0 \ge (2d)^{-1}\}$. Then, for $\a \ge \b$,  $\a \ge (d-1)\k - \f12$, 
\begin{align}\label{eq:main-lem}
\left|\int_{\CT^d} P_n^{(\a,\b)}(x_1 t_0 + x_2 t_1 + \cdots + x_d t_{d-1}) \xi(t)  (t_0 t_1 \cdots t_{d-1})^{k-1} 
   \dd t \right| \\
  \qquad\qquad\qquad 
    \le c n^{-2(d-1)\k} \frac{\prod_{j=2}^d |x_j - x_1|^{-\k}}{\big(\sqrt{1-|x_1|} + n^{-1}\big)^{\a+\f12 - (d-1)\k}}. \notag
\end{align}
 \end{lem}
 
\begin{proof}
With $t_0 = 1-t_1-\cdots - t_{d-1}$ and $\eta(t_1,\ldots,t_{d-1}) = t_0^{\k-1} \xi(t)$. Then $\eta$ is a 
$C^\infty(T^{d-1})$ function and the integral over $\CT^d$ can be written as 
\begin{align*}
I(x): = &\int_{\CT^d}  P_n^{(\a,\b)}(x_1 t_0 + x_2 t_1 + \cdots + x_d t_{d-1}) \xi(t) (t_0 \cdots t_{d-1})^{k-1} \dd t \\
= & \int_0^1  \int_0^{1-t_1}  \cdots \int_0^{1-t_1-\cdots - t_{d-1}} 
    P_n^{(\a,\b)}\big(x_1+ (x_2-x_1) t_1 + \cdots + (x_d-x_1) t_{d-1} \big)\\
    &  \qquad\qquad\qquad\qquad\qquad\qquad\quad 
    \times \xi(t) t_{d-1}^{k-1} \cdots t_2^{\k-1}  t_1^{\k-1} \dd t_{d-1} \cdots \dd t_1.
\end{align*}
Changing variable $t\mapsto u$ in the integral with 
$$
t_1 = u_1,  \quad t_2=(1- u_1) u_2,\quad  \ldots,  \quad t_{d-1} = (1-u_1)\cdots (1-u_{d-2}) u_{d-1},
$$
and write $\xi(t) = \xi(t(u))$. Then $t \in T^{d-1}$ becomes $u \in [0,1]^{d-1}$. It is easy to verify that 
$t_0(u) = (1-u_1) \cdots (1-u_{d-1})$, so that $\xi(t(u))$ is zero only if $(1-u_1)\cdots (1-u_{d-1}) 
\le 1/(2d)$. In particular, if $1-u_j \le 1/(2d)$, or $u_j \ge1- 1/(2d)$, for some $j$, then $\xi(t(u) = 0$. Let 
$\chi$ be a $C^\infty$ function such that $\chi(u) = 1$ if $u \le 1-1/(2d)$ and $\chi(u) = 0$ if $u \ge 1-1/(3d)$. 
Then $\chi(u_1) \cdots \chi(u_{d-1})$ is equal to 1 over the support set of $\xi(t(u))$. Hence, we can write
\begin{align*}
  I(x) = & \int_{[0,1]^{d-1}} P_n^{(\a,\b)}\bigg(x_1+\sum_{j=1}^{d-1} (x_{j+1}-x_1) \prod_{i=2}^{j} (1-u_{i-1}) u_j\bigg)
    \xi (t(u)) \\
     & \times   \prod_{j=1}^{d-1} \chi(u_{j}) u_j^{\k-1} \prod_{i=2}^{j}(1-u_{i-1})^\k \dd u_{d-1}\cdots \dd u_1.
\end{align*}
where we adopt the convention $\prod_{i=2}^1(1-u_{i-1})=1$. 
For $2 \le m \le d-1$, we define 
\begin{align*}
  f_m (y) =   \int_{[0,1]^{m}} & P_n^{(\a,\b)}\bigg(y + \sum_{j=d-m}^{d-1} (x_{j+1}-x_1) \prod_{i=2}^j(1-u_{i-1}) u_j\bigg)
    \xi (t(u)) \\
     & \times \prod_{j=d-m}^{d-1} \chi(u_j) u_j^{\k-1} \prod_{i=2}^j(1-u_{i-1})^\k   \dd u_{d-1}\cdots \dd u_{d-m}
\end{align*}
for all $y$ such that the argument of $P_n^{(\a,\b)}$ is in $[-1,1]$. It is evident that $I(x) = f_{d-1}(x_1)$.
Furthermore, let $v_{m}: = \lfloor |\a- m \k| + 2\k \rfloor +3$ for $0 \le m \le d-1$. For $\ell =0,1,\ldots, v_{m}$, 
we define 
\begin{align*}
F_{m,\ell} (y) = B_{n,\ell}
 \int_{[0,1]^{m}} & P_{n+\ell}^{(\a-\ell,\b-\ell)}\bigg(y + \sum_{j=d-m}^{d-1} (x_{j+1}-x_1)\prod_{i=2}^j(1-u_{i-1}) u_j\bigg)
    \xi (t(u)) \\
     & \times \prod_{j=d-m}^{d-1} \chi(u_j) u_j^{\k-1} \prod_{i=2}^j(1-u_{i-1})^\k   \dd u_{d-1}\cdots \dd u_{d-m},
\end{align*}
where $B_{n,\ell} =   2^{\ell}/ \prod_{j=1}^{\ell} (n+\a+\b+1-j)$. By \eqref{D-Jacobi}, it follows readily that 
$F_{m,\ell}^{(\ell)} = f_m$. We now prove that $f_m$ satisfies the estimate
\begin{equation} \label{eq:fm-bound}
   |f_m(y)| \le c\, n^{\a -2 m \k}  \prod_{j= d-m+1}^d |x_j - x_1|^{-\k} \big(1+n \sqrt{1-|y|}\big)^{-\a-\f12 + m \k}
\end{equation}
for $m=1,2,\ldots, d-1$ and, moroever, 
\begin{equation} \label{eq:fm-bound2}
    n^{-\a + 2 m \k}  \prod_{j= d-m+1}^d |x_j - x_1|^{\k} f_m \in \CS_n^{v_m}(\a - m \k). 
\end{equation}

The proof is by induction. For $m=1$, write $a_{d-1} = \prod_{i=2}^{d-1}(1-u_{i-1})$. Since $\chi$ 
is supported on $[0, b]$, where $b=1- (3d)^{-1})$, and $ n^{-\a}P_n^{(\a,\b)} (x) \in \mathcal{S}_n^{v_0}(\a)$ 
for $v_0 = \lfloor |\a|+2\k \rfloor + 3$, we can apply Lemma \ref{lem:ineq-main} to obtain
\begin{align*} 
 |f_1(y)|  & = \,  a_{d-1}^\k \left | \int_0^{1}  P_n^{(\a,\b)} \big(y + a_{d-1} (x_d-x_1)u_{d-1} \big) 
     \xi(t(u)) \chi(u_{d-1}) u_{d-1}^{\k-1} \dd u_{d-1}   \right|  \\
      &  \le c\, n^{\a -2\k} \left | x_d-x_1 \right|^{-\k} \left(1+n \sqrt{1-|y|}\right)^{-\a-\f12+\k},
\end{align*}
which establishes \eqref{eq:fm-bound} for $m=1$. Similarly, since $B_{n,\ell} \sim n^{-\ell}$, the similar estimate
can be carried out for $F_{1,\ell}$ and we obtain, for $\ell =1,\ldots, v_1$, 
$$
  n^{-\a} |F_{1,\ell}(y)| \le c \left|x_d-x_1 \right|^{-\k} n^{-2\k-2\ell} \left(1+n \sqrt{1-|y|}\right)^{-\a-\f12+\k+\ell},
$$
which shows that \eqref{eq:fm-bound2} holds for $m=1$. Assume now that \eqref{eq:fm-bound} and 
\eqref{eq:fm-bound2} have been established for $f_{m-1}$. We now consider $f_m$. From the
definition of $f_m$, it is easy to see that 
\begin{align*}
   f_{m}(y) =a_{d-m}^\k \int_0^1 & f_{m-1} \bigg( y + (x_{d-m+1}-x_1)a_{d-m} u_{d-m}\bigg)\xi (t(u)) \\
      & \times  \chi(u_{d-m}) u _{d-m}^{\k-1}  \dd u_{d-m},
\end{align*}
where $a_{d-m}= \prod_{i=2}^{d-m}(1-u_{i-1})$ for $m=2,3,\ldots, d-1$, and similar iterative relations hold 
if we replace $f_{m}$ and $f_{m-1}$ by $F_{m,\ell}$ and $F_{m-1,\ell}$ for $1 \le \ell \le v_m$. Using the 
induction hypothesis, we can then apply Lemma \ref{lem:ineq-main} to obtain the estimate \eqref{eq:fm-bound}
for $f_m$, which can be carried out exactly as in the estimate of $f_1$, and similarly for $F_{m,\ell}$ to 
establish \eqref{eq:fm-bound2} for $f_m$. This completes the induction. 

Finally, it is easy to see that the desired estimate \eqref{eq:main-lem} is equivalent to the estimate 
for $f_{d-1}(x_1)$ in \eqref{eq:fm-bound}. This completes the proof. 
\end{proof}

The support set of $\xi$ in the theorem means that we are considering the simplex with one vertex chopped off.
The proposition below shows that this can be done one at a time. 

\begin{prop} \label{prop:partition}
There exist $C^\infty$ functions $\xi_0, \xi_1,  \ldots,  \xi_{d-1}$ on $\CT^d$ such that
$$
 \xi_0(t) + \cdots +\xi_{d-1}(t) = 1, \qquad \xi_0(t) \ge 0, \, \ldots, \,  \xi_{d-1}(t) \ge 0, \qquad t \in \CT^d,
$$
and the support set of $\xi_j$ is a subset of $\{t \in \CT^d: t_j \ge (2d)^{-1}\}$. 
\end{prop}

\begin{proof}
For $d \ge 2$, let $0 < a < b < 1$ be defined by 
$$
    a = \frac{1}{2d} \quad \hbox{and} \quad b = \frac{1-a}{d-1} = \frac{1}{2d} + \frac{1}{2(d-1)} < 1. 
$$ 
Let $\xi$ be a $C^\infty$ function on the real line such that $\xi(t) = 0$ if $0 \le t \le a$, and $\xi(t) =1$ 
if $b < t \le 1$. In particular, $\xi$ is supported on $(a,1]$ and $1-\xi$ is supported on $[0, b)$. For $t 
\in \CT^d$, we write $t=(t_0, t_1,\ldots, t_{d-1})$ in homogeneous coordinates, or $t_0 = 1- t_1-\ldots - t_{d-1}$. 
We define $C^\infty$ functions $\xi_0, \xi_1,\ldots, \xi_{d-1}$ by 
\begin{align*}
  \xi_0(t) \, & = \xi(t_0),  \\
  \xi_1(t) \, & = \big(1-\xi(t_0) \big) \xi(t_1),  \\
     & \cdots \cdots \\
  \xi_{d-2}(t) \, & = \big(1-\xi(t_0) \big) \cdots \big(1-\xi(t_{d-3})\big) \xi(t_{d-2}),\\
  \xi_{d-1}(t) \, & = \big(1-\xi(t_0) \big) \cdots \big(1-\xi(t_{d-3})\big) \big(1-\xi(t_{d-2}) \big).
\end{align*}
Then it is evident that $\xi_0(t) + \ldots + \xi_{d-1}(t) = 1$. Furthermore, it is easy to see that,
for $0 \le j \le d-2$, the support set of $\xi_j$ is $\{t \in \CT^d: t_0 \le b, \ldots, t_j \le b, \, \hbox{and} \,
t_j > a\}$, which is evidently a subset of $\{t\in \CT^d: t_j > a\}$. Moreover, the support set of 
$\xi_{d-1}$ is $\{t \in \CT^d: t_0 \le b, \ldots, t_{d-2} \le b\}$. Each element of this last set satisfies 
the inequality
$$
    t_{d-1} = 1-t_0-\ldots - t_{j-2} \ge 1-(d-1) b = 1-(1-a) = a
$$
by the definition of $b$, so that the subset of $\xi_{d-1}$ is a subset of $\{t\in \CT^d: t_{d-1}>a\}$. 
This completes the proof. 
\end{proof}
 
\medskip
\noindent
{\it Proof of Theorem \ref{thm:estimate}.}
Using the partition of unity in Proposition \ref{prop:partition}, we can write the integral as a sum of 
$$
\int_{\CT^d} P_n^{(\a,\b)}(x_1 t_0 + x_2 t_1 + \cdots + x_d t_{d-1}) \xi_i(t) \varpi(t) (t_0 t_1 \cdots t_{d-1})^{k-1}\dd t 
$$
for $i =0,1,\ldots, d-1$. Hence, we only need to estimate the above integral for each $i$. For $i =0$, this 
is precisely the estimate carried out in Lemma \ref{lem:main} with $\xi(t) = \xi_0(t) \varpi(t)$. By the symmetry of 
$\CT_d$ and the integral, for each $i \ne 0$, we can exchange $t_i$ and $t_0$, so that the same estimate 
applies. This completes the proof. 
\qed

\end{document}